\documentclass[reqno]{amsart}
\usepackage{amssymb,amsmath,hyperref}
\usepackage{amsrefs, float}
\usepackage[foot]{amsaddr}
\usepackage{bbold,stackrel, multicol}
 
\newcommand{\Z}{{\textsf{\textup{Z}}}}

\newtheorem{thm}{Theorem}

\newtheorem{cor}[thm]{Corollary}
\newtheorem{defi}[thm]{Definition}

\newtheorem{rem}[thm]{Remark}
\newtheorem{nota}[thm]{Notation}

\newtheorem{princ}[thm]{Principle}

\newtheorem*{tempo*}{Template}

\newcommand\be{\begin{equation}}
\newcommand\ee{\end{equation}} 

\usepackage{amsmath,amsfonts} 

\usepackage[applemac]{inputenc}

\def\bdefi{\begin{defi}\rm}
\def\edefi{\end{defi}}
\def\bnota{\begin{nota}\rm}
\def\enota{\end{nota}}
\def\FIVE{\Pi_{1}^{1}\text{-\textup{\textsf{CA}}}_{0}}

\def\SIXko{\Pi_{k+1}^{1}\text{-\textsf{\textup{CA}}}_{0}}
\def\SIXK{\Pi_{k}^{1}\text{-\textsf{\textup{CA}}}_{0}^{\omega}}

\def\ATR{\textup{\textsf{ATR}}}

\def\ZF{\textup{\textsf{ZF}}}

\def\SUP{\textup{\textsf{sup}}}

\def\RCA{\textup{\textsf{RCA}}}
\def\({\textup{(}}
\def\){\textup{)}}

\def\RCAo{\textup{\textsf{RCA}}_{0}^{\omega}}
\def\ACAo{\textup{\textsf{ACA}}_{0}^{\omega}}

\def\N{{\mathbb  N}}
\def\Q{{\mathbb  Q}}
\def\R{{\mathbb  R}}

\def\SS{\textup{\textsf{S}}}
\def\FSAC{\textup{\textsf{finite-$\Sigma_{1}^{1}$-AC$_{0}$}}}

\def\di{\rightarrow}

\def\asa{\leftrightarrow}

\def\ACA{\textup{\textsf{ACA}}}

\def\QFAC{\textup{\textsf{QF-AC}}}

\def\SAC{\Sigma_{1}^{1}\textup{\textsf{-AC}}_{0}}

\def\GS{\textup{\textsf{GS}}}
\def\KINL{\textup{\textsf{KINL}}}

\def\SIND{\Sigma\textup{\textsf{-IND}}}

\def\SUP{\textup{\textsf{sup}}}
\def\SAC{\textup{\textsf{$\Sigma_{1}^{1}$-AC$_{0}$}}}

\def\cocode{\textup{\textsf{cocode}}}

\def\u{\textup{\textsf{u}}}

\def\BOOT{\textup{\textsf{BOOT}}}

\def\IND{\textup{\textsf{IND}}}

\def\RANGE{\textup{\textsf{RANGE}}}

\def\CWO{\textup{\textsf{CWO}}}

\def\SUP{\textup{\textsf{sup}}}

\usepackage{mathrsfs}  

\usepackage{graphicx}
\usepackage{tikz}
\usetikzlibrary{matrix, shapes.misc}
\usepackage{comment,tikz-cd}

\numberwithin{equation}{section}
\numberwithin{thm}{section}

\begin{document}
\title{On countability and representations}
%\author{Dag Normann}
%\address{Department of Mathematics, The University 
%of Oslo, P.O. Box 1053, Blindern N-0316 Oslo, Norway}
%\email{dnormann@math.uio.no}
\author{Sam Sanders}
\address{Department of Philosophy II, RUB Bochum, Germany}
\email{sasander@me.com}

\begin{abstract}
The topic of this paper is the subtle interplay between countability and representations.  
In particular, we establish that the definition of countability of a certain set $X$ crucially hinges on the associated equivalence relation $=_{X}$.  
Armed with this knowledge, we study well-known and basic principles about countable sets, going back to Cantor, Sierpi\'nski, and K\"{o}nig, working in Kohlenbach's higher-order Reverse Mathematics.
While these principles are relatively weak in second-order Reverse Mathematics, we obtain equivalences involving countable choice and Feferman's projection principle.  
The latter are essentially the strongest axioms studied in higher-order Reverse Mathematics and usually only come to the fore when dealing with the uncountable.       
\end{abstract}
%
%\setcounter{page}{0}
%\tableofcontents
%\thispagestyle{empty}
%\newpage
%%
\maketitle              % typeset the header of the contribution
%
%\vspace{-0.5mm}
\section{Introduction}
\subsection{Aim and motivation}\label{aimo}
In a nutshell, we study the definition of `countable set $X$' so as to bring out the crucial role of the associated equivalence relation `$=_{X}$'.  
With the latter in place, we show that basic principles about countable sets imply strong axioms, even compared to known results concerning the uncountable.  

\smallskip

In more detail, the central question we study in this paper is both exceedingly simple and devilishly subtle.  
\begin{center}
\emph{What is a countable set?}
\end{center}
The most general answer is of course that a set be countable if there is an injection to the naturals.  
By definition, such a mapping only sends \emph{equal} objects to \emph{equal} naturals.  
While equality on the naturals should be clear, equality for other objects is a (much) more 
subtle affair in case the latter involve \emph{representations}.  

\smallskip

For instance, real numbers can be represented via Cauchy sequences.  
Thus, two reals $x, y$ are called equal, denoted $x=_{\R}y$, if the underlying Cauchy sequences have the same limit.  
However, there are uncountably many different Cauchy sequences that converge to $0$.  Hence, a singleton set $\{x\}$ contains
only one real number, namely $x$, but uncountably many different representations of this real $x$.  

\smallskip

More generally, in a metric space $(M, d)$, two objects $x, y\in M$ are equal, denoted $x=_{M} y$, if $d(x, y)=_{\R}0$, by the definition of metric space.   
Hence, the observations from the previous paragraph carry over.  Moreover, in the Lebesgue space $L^{p}$, two functions remain equal (under the induced metric) if we only change the former on a measure zero set. 
The point is again that a single element of a metric space $(M, d)$ can have \emph{many different} representations that are all equal under $=_{M}$.  

\smallskip

Now, the previous paragraphs motivate the following definition of countability.  
\begin{defi}\label{trifff}
A set $X$ with equivalence relation $=_{X}$ is called \emph{countable} if there is a mapping $Y:X\di \N$ such that 
\be\label{edelkraft}
(\forall x, y\in X)(Y(x)=_{\N}Y(x)\di x=_{X}y).
\ee
\edefi
Our motivation for all the above pedantry is as follows:  we show in Section~\ref{defguy} that, when formulated with Definition \ref{trifff}, basic and well-known principles about countable sets of reals imply or are equivalent to rather strong axioms, working in Kohlenbach's higher-order Reverse Mathematics (often abbreviated `RM').  We shall show that well-known theorems by Cantor, Sierpi\'nski, and K\"ong imply Feferman's projection principle and/or countable choice. 

\smallskip

We introduce definitions and preliminaries regarding Kohlenbach's higher-order approach to RM in Section~\ref{prelim}.  
We stress that the aforementioned axioms, countable choice and Feferman's projection principle, are the strongest generally\footnote{The only exception known to the author is a version of the Axiom of \emph{dependent} Choice akin to $\QFAC^{0,1}$ studied in \cite{dagsam16}.} encountered in higher-order RM, especially for theorems dealing with countable sets.  

\smallskip

In terms of second-order consequences, we can formulate e.g.\ Cantor's characterisation theorem for countable linear orders $(X, \preceq_{X})$ using either Definition \ref{trifff} for $X\subset \R$ or `$X$ is a countable subset of $\R$'.  
As we will see, the latter implies $\FIVE$ while the former remains at the level of $\ATR_{0}$.  %The second-order version is equivalent to $\ACAo$ (\cite{simpson2})

\smallskip

We discuss and motivate Definition \ref{trifff} in more detail in Section \ref{othernamez}.  
We study the following topics in Section \ref{tonyfox}-\ref{pifpoefpaf}.  
\begin{itemize}  
\item Cantor's characterisations of countable linear orders and related statements about orderings (Section \ref{tonyfox}).  
\item Sierpi\'nski's characterisation of countable metric spaces without isolated points (Section \ref{countri}).
\item The supremum principle for continuous functions on countable second-countable spaces and the Ginsburg-Sands theorem (Section \ref{SPAST}).  
\item K\"onig's infinity lemma and similar principles (Section \ref{pifpoefpaf}).  
\end{itemize}
In each case, we obtain equivalences involving Feferman's projection principle and/or countable choice.  %The necessary definitions are introduced in Section \ref{prelim}.   
The Ginsburg-Sands theorem has recently been studied in second-order RM (\cite{benham}).  

\smallskip

In conclusion, Definition \ref{trifff} constitutes a most general notion of countability while basic principles yield strong axioms when formulated with the former.  
%The foundational implications of this observation are discussed in Section \ref{after}.  

\subsection{Preliminaries}\label{prelim}
\subsubsection{Introduction}
We introduce some required definitions for the below.  We assume familiarity with Kohlenbach's higher-order RM, the base theory $\RCAo$ in particular.  
The original reference is \cite{kohlenbach2} with a more `basic' extended introduction in \cite{sammetric}.   A monograph on higher-order RM is currently under review (\cite{samBOOK}).
For the rest of this section, we describe the essentials needed for developing real analysis in higher-order RM.   

\smallskip

First of all, real numbers are defined in $\RCAo$ in the same way as in second-order RM, i.e.\ as fast-converging Cauchy sequences (\cite{simpson2}*{II.5.3}).
Hence, the meaning of `$x=_{\R}y$' and `$x\in \R$' is clear while a function from reals to reals is simply a mapping $\Phi:\N^{\N}\di \N^{\N}$ mapping equal reals to equal reals.  
\be\label{sefk}
 (\forall x, y \in \R)(x=_{\R}y\di \Phi(x)=\Phi(y)).  
 \ee
The formula \eqref{sefk} is often described as `$\Phi$ satisfies the axiom of real extensionality' and variations. 
We use $f:\R\di \R$ to suggest that $f$ is a function from $\R$ to $\R$.  
 
 \smallskip
 
 Secondly, having defined functions on the reals, we can define sets of reals.  
 \bdefi[Sets]\label{char}
 \begin{itemize}
 \item A subset $X\subset \R$ is defined via $F_{X}:\R\di \{0,1\}$ where $x\in X\asa F_{X}(x)=1$ for all $x\in \R$. 
\item We write `$A\subseteq B$' if we have $(\forall x\in \R)(x\in A\di x\in B)$.  % for all $x\in \R$. 
\item A set $O\subseteq \R$ is \emph{open} in case $x\in O$ implies that there is $k\in \N$ such that $B(x, \frac{1}{2^{k}})\subseteq O$.\label{qzopen}  
\item The complement $O^{c}$ of an open set $O\subset \R$ is called closed.  
\item A set $X\subset \R$ is \emph{enumerable} if there is a sequence $(x_{n})_{n\in \N}$ of reals that includes all elements of $X$.  
 \end{itemize}
 \edefi
By the previous, sets are given by characteristic functions and we shall always work in a strong enough logical system (see Section \ref{eenbeetjeaxi})
to guarantee that e.g.\ the unit interval is a set as in Definition \ref{char}.  The definition of `countable set of reals' is discussed in detail in Section \ref{othernamez}.  
 
\smallskip

Finally, we sometimes use type theoretic notation, i.e.\ $n^{0}$ for $n\in \N$, $f^{1}$ for $f\in \N^{\N}$, or $Y^{2}$ for $Y:\N^{\N}\di\N$.  
In this way, finite sequences of naturals and elements of Baire space are denoted $v^{0^{*}}$ and $w^{1^{*}}$; the number $|u|$ is the length of the finite sequence $u$ while $u(i)$ is the $i$-th element of $u$ for $i<|u|$.  
 
\subsubsection{Axioms of higher-order arithmetic}\label{eenbeetjeaxi}
We introduce some functionals that serve as the axioms of higher-order RM.  
We stress that some of the below functionals were already 
studied by Hilbert and Bernays in the \emph{Grundlagen} (\cite{hillebilly2}).  

\smallskip

First of all, we consider the following axiom where the functional $E$ is also called \emph{Kleene's quantifier $\exists^{2}$} and is discontinuous on Cantor space.  
\be\tag{$\exists^{2}$}
(\exists E:\N^{\N}\di \{0,1\})(\forall f\in\N^{\N})( E(f)=0 \asa (\exists n\in \N)(f(n)=0)).
\ee
We write $\ACAo\equiv \RCAo+(\exists^{2})$ and observe that the latter proves the same second-order sentences as $\ACA_{0}$ (see \cite{hunterphd}).
We shall mostly work in $\ACAo$ as this guarantees that the unit interval is a set following Definition \ref{char}.  
Over $\RCAo$, $(\exists^{2})$ is equivalent to $(\mu^{2})$ (\cite{kohlenbach2}) where the later expresses that there is $\mu : \N^{\N}\di \N$ such that for $f\in \N^{\N}$ we have
\be\label{muf}
(\exists n\in \N)(f(n)=0)\di f(\mu(f))=0.
\ee
%We are mostly dealing with \emph{conventional} comprehension here, i.e.\ only parameters over $\N$ and $\N^{\N}$ are allowed in formula classes like $\Pi_{k}^{1}$ and $\Sigma_{k}^{1}$.  
Secondly, consider the following axiom where the functional $\SS^{2}$ is often called \emph{the Suslin functional} (\cite{kohlenbach2, avi2, yamayamaharehare}):
\be\tag{$\SS^{2}$}
(\exists\SS:\N^{\N}\di \{0,1\})(\forall f \in \N^{\N})\big[  (\exists g \in \N^{\N})(\forall n \in \N)(f(\overline{g}n)=0)\asa \SS(f)=0  \big].
\ee
By definition, the Suslin functional $\SS^{2}$ can decide whether a $\Sigma_{1}^{1}$-formula in normal form, i.e.\ as in the left-hand side of $(\SS^{2})$, is true or false.   
The system $\FIVE^{\omega}\equiv \RCAo+(\SS^{2})$ proves the same $\Pi_{3}^{1}$-sentences as $\FIVE$ (see \cite{yamayamaharehare}).  

\smallskip

Thirdly, we define the functional $\SS_{k}^{2}$ which decides the truth or falsity of $\Sigma_{k}^{1}$-formulas in normal form; we also define 
the system $\SIXK$ as $\RCAo+(\SS_{k}^{2})$, where  $(\SS_{k}^{2})$ expresses that $\SS_{k}^{2}$ exists.  
We define $\Z_{2}^{\omega}$ as $\cup_{k}\SIXK$ as one possible higher-order version of $\Z_{2}$.
The functionals $\nu_{n}$ from \cite{boekskeopendoen}*{p.\ 129} are essentially $\SS_{n}^{2}$ strengthened to return a witness (if existent) to the $\Sigma_{n}^{1}$-formula at hand.  %  if it exists. 
The operator $\nu_{n}$ is essentially Hilbert-Bernays' operator $\nu$ from \cite{hillebilly2}*{p.\ 479} restricted to $\Sigma_{n}^{1}$-formulas. 

\smallskip

\noindent
Fourth, we introduce Kleene's quantifier $\exists^{3}$ as follows:
\be\tag{$\exists^{3}$}
(\exists E: (\N^{\N}\di \N)\di \N)(\forall Y:\N^{\N}\di \N)\big[  (\exists f \in \N^{\N})(Y(f)=0)\asa E(Y)=0  \big].
\ee
Both $\Z_{2}^{\Omega}\equiv \RCAo+(\exists^{3})$ and $\Z_{2}^\omega\equiv \cup_{k}\SIXK$ are conservative over $\Z_{2}$ (see \cite{hunterphd}).  
%Despite this close connection, $\Z_{2}^{\omega}$ and $\Z_{2}^{\Omega}$ can behave quite differently, as will become clear.   
The functional from $(\exists^{3})$ is also called `$\exists^{3}$', and we use the same convention for other functionals.  Hilbert-Bernays' operator $\nu$ (see \cite{hillebilly2}*{p.\ 479}) is essentially Kleene's $\exists^{3}$, modulo a non-trivial fragment of the Axiom of (quantifier-free) Choice.    

\smallskip

Fifth, the late Sol Feferman was a leading Stanford logician with a life-long interest in foundational matters (\cite{fefermanlight, fefermanmain}), especially so-called predicativist mathematics following Russell and Weyl (\cite{weyldas}). 
As part of this foundational research, Feferman introduced the `projection principle' \textsf{Proj}$_{1}$ in \cite{littlefef}, a third-order version of $(\exists^{3})$.
%% that is impredicative and highly explosive\footnote{Feferman's predicative system $\textup{VFT}+(\mu^{2})$ from \cite{littlefef} is similar to $\ACAo$.  The former plus  \textsf{Proj1} implies $\Z_{2}$, as noted in \cite{littlefef}*{\S5}.}.
%%In this section, we obtain equivalences between our version of $\textsf{Proj1}$ and the supremum principle (Princ.~\ref{SIP}) for weak continuity classes from Def.\ \ref{KY}.  
We study the principle $\BOOT$, which is \textsf{Proj}$_{1}$ in the language of higher-order RM
\begin{princ}[$\BOOT$] For $Y^{2}$, there is $X\subset \N$ such that 
\be\label{EZ}
(\forall n\in \N)\big[n\in X\asa (\exists f\in \N^{\N})(Y(f, n)=0)  \big].
\ee
\end{princ}
The name of this principle derives from the verb `to bootstrap'.  Indeed, $\BOOT$ is third-order and weaker than $(\exists^{3})$, but we still have that $\SIXK+\BOOT$ proves $\SIXko$.
In other words, $\BOOT$ bootstraps itself along the second-order comprehension hierarchy so central to mathematical logic (see \cite{sigohi}).  
%oth $\BOOT$ and $(\exists^{3})$ are equivalent to (intimately related) basic statements about continuous functions on metric spaces, as studied in Chapter \ref{metric}.  
We have previously obtained equivalences for $\BOOT$ (\cite{samHARD, samSECOND}) but the results in Section \ref{defguy} are a massive improvement as they 
pertain to basic statements about \emph{countable} objects.  

\smallskip %\newpage

Next, the following fragment of countable choice, not provable in $\ZF$, plays an important role in higher-order RM.
\begin{princ}[$\QFAC^{0,1}$]
Let $\varphi$ be quantifier-free with $(\forall n\in \N)(\exists f\in \N^{\N})\varphi(f, n)$, then there exists a sequence $  (f_{n})_{n\in \N}$ in $\N^{\N}$ with $(\forall n\in \N)\varphi(f_{n}, n)$.
\end{princ}
The local equivalence between sequential and `epsilon-delta' continuity cannot be proved in $\ZF$ (\cite{heerlijkheid}), but can be established in $\RCAo+\QFAC^{0,1}$ (\cite{kohlenbach2}).  

\smallskip

We recall that certain equivalences in second-order RM require induction beyond what is available in $\RCA_{0}$ (see \cite{neeman}). 
The following fragment of the induction axiom is among the strongest used in higher-order RM; the numbering is taken from \cite{samBOOK}.
\begin{princ}[$\IND_{3}$]
The induction axiom for formulas $\varphi(n)$ of the form $(\exists f\in \N^{\N})(\forall g\in \N^{\N})(Y(f,g,  n)=0)$ for any $Y^{2}$.  
\end{princ}
We let $\SIND$ be the previous schema with the universal quantifier `$(\forall g\in \N^{\N})$' omitted.  
Like in second-order RM (see \cite{simpson2}*{X4.4}), induction yields bounded comprehension in that $\IND_{3}$ readily\footnote{Apply $\SIND$ to $\varphi(k)$ defined as $(\exists X\subset \N)(\forall n\leq k)(n\in X\asa  (\exists f\in \N^{\N})(Y(f, n)=0))$} implies the following principle. 
\begin{princ}[Bounded comprehension]
For $Y^{2}$ and $k\in \N$, there is $X\subset \N$ with $(\forall n\leq k)(n\in X\asa (\exists f\in \N^{\N})(Y(f, n)=0))$. 
\end{princ}
Finally, we finish this section with a remark on the above conservation results.  
\begin{rem}[Hunter's conservation results]\label{texxchno}\rm
Hunter proves a number of interesting and useful conservation results in \cite{hunterphd}.  We are interested in the proof that $\ACAo+\QFAC^{0,1}$ is  conservative over $\SAC$, 
In Hunter's proof, an arbitrary model $\mathcal{M}$ of $\SAC$ is extended to a (term) model $\mathcal{N}$ of $\ACAo+\QFAC^{0,1}$, with the second-order part of $\mathcal{N}$ isomorphic to $\mathcal{M}$.  
The completeness theorem then yields the required conservation result.  That $\mathcal{N}$ satisfies $\QFAC^{0,1}$ follows from:
\begin{itemize}
\item[(A)] the original model $\mathcal{M}$ and the second-order part of $\mathcal{N}$ satisfy $\SAC$,  
\item[(B)] arithmetical formulas with higher-order parameters are equivalent (in $\mathcal{N}$) to arithmetical formulas involving only second-order parameters (from $\mathcal{M}$).  
\end{itemize}
The items (A) and (B) provide a kind of template for other conservation results.  Indeed, consider Principle \ref{weeralweer} right below, 
which is equivalent to $\ATR_{0}$ when restricted to second-order parameters (see \cite{simpson2}*{V.5.2}).
\begin{princ}\label{weeralweer}
For arithmetical $\varphi$ with $(\forall n\in \N)(\exists \textup{ at most one } X\subset \N)\varphi(X, n)$, there is $Z\subset \N$ with $n\in Z\asa (\exists X\subset \N)\varphi(X, n)$, for all $n\in \N$. 
\end{princ}
To show that $\ACAo$ plus Principle \ref{weeralweer} (involving higher-order parameters) is conservative over $\ATR_{0}$, one simply starts with a model $\mathcal{M}$ of the latter and considers Hunter's term model $\mathcal{N}$. 
The latter is a model of $\ACAo$ by construction and of $\ATR_{0}$ by assumption.  Since $\ATR_{0}$ is equivalent to the second-order version of Principle \ref{weeralweer}, we may use item (B) above to guarantee that $\mathcal{N}$ satisfies Principle~\ref{weeralweer} \emph{involving higher-order parameters}, as required.   
\end{rem}
Hunter's conservation results and variations will be included in \cite{samBOOK}*{Appendix}, with permission of the original author.  

\subsubsection{Spaces, metric and otherwise}
For completeness, we introduce the definitions of metric and second-countable space used in higher-order RM. 

\smallskip

First of all, we introduce the well-known definition of metric space $(M, d)$ to be used in this monograph, namely Definition \ref{donkc}.  
The latter is really the textbook definition with some details like function extensionality made explicit. 
We shall generally only study the case where $M$ is a subset of $\R$, up to coding.  

\smallskip

Now, in our study of metric spaces $(M, d)$, we assume that the set $M$ comes with its own equivalence relation `$=_{M}$' and that the metric $d:M^{2}\di \R$ satisfies 
the axiom of extensionality relative to this relation as follows:
\[
(\forall x, y, v, w\in M)\big([x=_{M}y\wedge v=_{M}w]\di d(x, v)=_{\R}d(y, w)\big).
\]
Similarly to functions on the reals, `$F:M\di \R$' denotes a function that satisfies the following instance of the axiom of function extensionality:
\be\tag{\textup{\textsf{E}}$_{M}$}\label{koooooo}
(\forall x, y\in M)(x=_{M}y\di F(x)=_{\R}F(y)).
\ee
We recall that the study of metric spaces in second-order RM is at its core based on equivalence relations, as discussed explicitly in e.g.\ \cite{simpson2}*{I.4.3} or \cite{damurm}*{\S10.1}.   
%i.e.\ function extensionality relative to $M$.  

\smallskip

Secondly, we now have the following definition of metric space. 
\bdefi[$\RCAo$]\label{donkc}
A functional $d: M^{2}\di \R$ is a \emph{metric on $M$} if it satisfies the following properties for $x, y, z\in M$:
\begin{enumerate}
 \renewcommand{\theenumi}{\alph{enumi}}
\item $d(x, y)=_{\R}0 \asa  x=_{M}y$,
\item $0\leq_{\R} d(x, y)=_{\R}d(y, x), $
\item $d(x, y)\leq_{\R} d(x, z)+ d(z, y)$.
\end{enumerate}
\edefi
To be absolutely clear, we shall study metric spaces $(M, d)$ with $M\subset \N^{\N}$ or $M\subset \R$, unless explicitly stated otherwise. 
Thus, quantifying over $M$ amounts to quantifying over $\N^{\N}$ or $\R$, perhaps modulo coding of finite sequences, i.e.\ the previous definition can be made in third-order arithmetic.   
%Since we shall mostly study \emph{compact} metric spaces, this restriction is minimal as the cardinality of such spaces is at most that of the continuum by \cite{buko}*{Theorem 3.13}.

\smallskip

Thirdly, we introduce the well-known definition of second-countable space to be used in the below, namely as follows. 
%Secondly, our definition of second-countable spaces is given by Definition \ref{CSC}.  
\bdefi[Second-countable spaces]\label{CSC}~
\begin{itemize}
\item A \emph{basis/base} for a topology on $X$ is a collection $(U_{i})_{i\in I}$ and a mapping $k:(X\times I^{2})\di I$ satisfying the following.
\begin{itemize}
\item For every $x\in X$, there is $i\in I$ with $x\in U_{i}$.
\item For $x\in X$ and $i,j\in I$, we have $x\in U_{i}\cap U_{j}\di x\in U_{k(x, i, j)}\subseteq U_{i}\cap U_{j}$.
\end{itemize}
\item A \emph{countable second-countable space} consists of a set $X\subset \N$ and a basis with index set $I=\N$.  We abbreviate this by `CSC-space $X$'.  
\item A \emph{real second-countable space} consists of a set $X\subset \R$ and a basis with index set $I=\N$.  We abbreviate this by `RSC-space $X$'.  
%\item For a second-countable space $X$ with basis $(U_{i})_{i\in \N}$ and $x, y\in X$, we say that `$x$ and $y$ are \(topologically\) indistinguishable' in case $x\in U_{i}\asa y\in U_{i}$ for all $i\in \N$.  We then write `$x\equiv_{X} y$'.
%\item An RSC-space $X$ is \emph{\(strongly\) countable} if $X$ is \(strongly\) countable.  
%a second-countable space that is countable, i.e.\ there is $Y:X\di \N$ such that $(\forall x, y\in X\)( Y(x)=Y(y)\di x=_{\R} y )$.
\end{itemize}
\edefi
Below, we only study second-countable spaces with size at most the continuum, i.e.\ RSC-spaces.
We note that CSC-spaces are second-order objects and have been studied in second-order RM (\cites{dor1,damurm, benham}).  
On a technical note, the mapping $k$ in the definition is included as the latter is generally used in second-order RM.
%Moreover, a `countable RSC space' amounts to a CSC-space with `enumerable set $X\subset \N$' replaced by `countable set $X\subset \R$'. 
%Fo

\section{Countability and representations}\label{defguy}
\subsection{Introduction}\label{gintro}
We establish the subtle interplay between countability and representations.  
In particular, we observe in Section~\ref{othernamez} that the definition of countability of a certain set $X$ crucially hinges on the associated equivalence relation $=_{X}$.  
Armed with this knowledge, we obtain new equivalences for $\BOOT$ and $\QFAC^{0,1}$ for basic principles concerning countable objects (Sections \ref{tonyfox}-\ref{pifpoefpaf}).  

\smallskip

We stress that the combination $\BOOT+\QFAC^{0,1}$ is strong even compared to the `usual' axioms encountered in higher-order RM.    
The principles studied in \cites{dagsamV, dagsamVII, dagsamVI, dagsamX, dagsamXI, samBIG, samBIG2, samBIG3} are much weaker.  
Moreover, properties of \emph{countable} sets of reals generally follow from principles (much) weaker than $\BOOT$, as discussed in Section \ref{tonyfox}.

\smallskip

The principles studied in the below merely constitute examples: we could obtain many more equivalences but have elected to study third-order versions of theorems 
already studied in second-order RM (in some version or other).   

\smallskip

Finally, we note that in Bishop's constructive analysis, a set comes with an equality notion as is clear from this quote from \cites{bish1,bridge1}.  
\begin{quote}
A set is not an entity which has an ideal existence. A set exists only
when it has been defined. To define a set we prescribe, at least implicitly,
what we (the constructing intelligence) must do in order to construct
an element of the set, and what we must do to show that two elements
of the set are equal.
\end{quote}
The connection between RM and constructive mathematics (see e.g.\ \cite{kohlenbach2}) is well-known and our results in this paper 
can be interpreted as providing another link.  

\smallskip

\subsection{Countability by any other name}\label{othernamez}
We recall the definition of countability of sets in Baire space and related spaces and observe the essential nature of the associated equivalence relation.  

\smallskip

First of all, we recall the definition of countability for subsets of $\N^{\N}$, observing that elements of the latter do not involve representations.
In particular, $A\subset \N^{\N}$ is called \emph{countable} if there exists $Y:\N^{\N}\di \N$ that is injective on $A$, i.e.\ we have
\be\label{INJ1}
(\forall f, g\in A)( Y(f)=_{\N}Y(g)\di f=_{1}g ).
\ee
Thus, in case $Y$ maps $f, g\in A$ to the same natural number, then $f$ and $g$ must be the same object in the Baire space, i.e.\ $f=_{1}g$, which means $(\forall n\in \N))(f(n)=_{\N}g(n))$.  

\smallskip

Secondly, we similarly consider the definition of countability for subsets of $\R$, recalling that elements of the latter \emph{do} involve representations, namely as fast-converging Cauchy sequences.
Now, a set $A\subset \R$ is called \emph{countable} if there exists $Y:\R\di \N$ that is injective on $A$, i.e.\ we have
\be\label{INJ2}
(\forall x, y\in A)( Y(x)=_{\N}Y(y)\di x=_{\R}y ).
\ee
We stress the role of the consequent in \eqref{INJ2}:  in case $Y$ maps $x, y\in A$ to the same natural number, then $x$ and $y$ must (only) be the same real number. 
In particular, $x$ and $y$ may be different representations of the same real number, i.e.\ we must have $x=_{\R} y$ but $x\ne_{1}y$ is possible. 

\smallskip

Thirdly, each real number has uncountably many different representations and a set $A\subset \R$ satisfying \eqref{INJ2} is therefore countable \emph{from the point of view of $\R$}, but does contain uncountably many different elements of Baire space, divided among countably many equivalence classes provided by $Y$ as in \eqref{INJ2}.  We stress the essential role of `$=_{\R}$' in the definition of countability as in \eqref{INJ2}:  the equality `$=_{\R}$' is crucial in defining a meaningful notion of countability.  

\smallskip

Fourth, we consider metric spaces $(M, d)$ where we recall that $x=_{M}y$ is equivalent to $d(x, y)=_{\R}0$ by the definition of metric.  
The observation from the previous paragraph of course motivates our definition of countability.  Indeed, a set $A\subset M$ in a metric space $(M, d)$ is countable in case there exists $Y:M\di \N$ such that 
\be\label{INJ3}
(\forall x, y\in A)( Y(x)=_{\N}Y(y)\di x =_{M}y), 
\ee
where we stress the essential role of `$=_{M}$' in the consequent.   
Like for the real numbers, the definition of countability involves countably many equivalence classes that (can) each have uncountably many members that are equal under $=_{M}$. 
To be absolutely clear, we adopt \eqref{INJ3} as it mirrors \eqref{INJ2}, its use of the equivalence relation $=_{M}$ in particular.  

\smallskip

Fifth, we point out a crucial difference between countability for subsets of $\R$ and countability for subsets of metric spaces:  given $\ACAo$, $\R$ comes with a countable dense sub\emph{sequence}, the rationals $\Q$, which provides a method for converting reals into binary representations, i.e.\ each equivalence class has a canonical representative.   By contrast, compact metric spaces do not come with such approximation devices in $\ACAo$ (and much stronger systems) by the results in \cite{samHARD}.         

\smallskip
 
In conclusion, the definitions of countability for subsets of the reals or of metric spaces as in \eqref{INJ2} or \eqref{INJ3} crucially hinges on the associated equivalence relation, namely $=_{\R}$ and $=_{M}$.  
This motivates our use of Definition \ref{trifff}.  %A similar observation can of course be made for \emph{finite} sets, which is the topic of Sect. 
%We do not claim originality here: Bishop (\cite{bish1}) already observes the importance of sets coming with an equality relation.  
 %
%\smallskip

 %\subsection{Graph theory and related areas}
 
\subsection{Countable linear orders}\label{tonyfox}
We study the RM-properties of linear orders that are countable in the sense of Definition \ref{trifff}, often going back to Cantor.    In particular, we shall obtain equivalences involving $\BOOT+\QFAC^{0,1}$.     
 
\smallskip

First of all, we have previously defined a linear order $(X, \preceq_{X})$ to be countable if the set $X\subset \R$ has this property (as a subset of reals); with this natural definition, Cantor's characterisation of countable linear orders 
is equivalent to $\cocode_{0}$ (\cite{dagsamXI}).  
\begin{princ}[$\cocode_{0}$]
For any countable $A\subset \R$, there is a sequence $(x_{n})_{n\in \N}$ that includes all elements of $A$.  
\end{princ}
Now, $\cocode_{0}$ boasts many equivalences, including Principle \ref{weeralweer} (involving higher-order parameters) and the Jordan decomposition theorem (see \cites{samBIG, dagsamXI}).  
Hence, $\ACAo+\cocode_{0}$ is a conservative extension of $\ATR_{0}$ by Remark \ref{texxchno}.  By contrast, $\ACAo+\BOOT$ readily proves $\FIVE$.    

\smallskip

Secondly, the following definition of countable linear order will turn out to be more general than simply considering countable sets of reals.  
We use the usual\footnote{Namely that the relation $\preceq_{X}$ is transitive, anti-symmetric, and connex, like in \cite{simpson2}*{V.1.1}.} definition of \emph{linear ordering} where `$\preceq_{X}$' is given by a characteristic function $F_{X}:\R^{2}\di \{0,1\}$, i.e. `$x\preceq_{X} y$' is shorthand for $F_{X}(x, y)=_{\N}1$.  %while $(X, \preceq_{X})$ is called \emph{\(height\) countable} if $X\subset \R$ is.  

\bdefi[CLO]\label{CLODEF}
 For $X\subset \R$, we say that 
\emph{$(X, \preceq_{X})$ is a CLO} in case $(X, \preceq_{X})$ is a linear order and there is $Y:\R\di \N$ such that 
\be\label{INJU}
(\forall x, y\in X)(Y(x)=_{\N}Y(x)\di x=_{X} y). 
\ee
\edefi
%We have previously defined a linear order $(X, \preceq_{X})$ to be (height-)countable if the set $X\subset \R$ has this property (see Section \ref{dacantor}).  
%While this choice seems natural and yields nice equivalences involving $\enum$, Definition \ref{CLODEF} is more general.  
%We stress the final `$=_{X}$'  in \eqref{INJU} as this use of equivalence classes constitutes the fundamental difference compared to `$X$ is a countable set of reals'.  
We point out that the use of equivalence classes for linear orders is part and parcel\footnote{In particular, consider the definition of `countable ordinal' from \cite{simpson2}*{V.2.10}; in the latter, two countable ordinals are defined to be \emph{equal} if they are isomorphic.  With this definition, $\ATR_{0}$ is equivalent to the statement \emph{the countable linear orders form a linear order} by \cite{simpson2}*{V.6}.} of second-order RM.  

\smallskip

%Secondly, we consider the following theorem where the second item is Cantor's characterisation theorem for CLOs.  
Thirdly, an order-isomorphism from $(X, \preceq_{X})$ to $(Y, \preceq_{Y})$ is a surjective\footnote{Note that \eqref{crange} implies that $(\forall x,x' \in X)( Z(x)=_{Y}Z(x') \di x=_{X}x' )$, i.e.\ $Z$ is injective relative to the equalities `$=_{X}$' and `$=_{Y}$', i.e.\ `surjective' may be replaced by `bijective'.} $Z: X\di Y$ that respects the order relation (see \cite{simpson2}*{Def.\ V.2.7}), i.e.\
\be\label{crange}
(\forall x,x' \in X)(x\preceq_{X} x' \asa Z(x)\preceq_{Y}Z(x')).
\ee
% The reader should verify that using a stronger definition of order-isomorphism does not change the below equivalences. 
In this context, Cantor contributes item (b) from Theorem \ref{kakati} for any countable set, as discussed in \cite{riot}*{p.\ 122-123}. 
Moreover, Cantor introduces the notion of \emph{order type} in \cite{cantorbook90} and characterises the order type $\eta$ of $\Q$ in \cite{cantorm} based on item (c).
%\end{rem}
%The previous remark leads to the following theorem, suggesting that CLOs are much more general than `linear orders that are (height-)countable'. 
%\begin{rem}[$\ACAo$]\rm
%The principle $\BOOT$ follows from the following.  
%\begin{center}
%\emph{A CLO $(X, \preceq_{X})$ with $X\subset \R$ is order-isomorphic to a subset of $\Q$.}
%\end{center}
%We provide a solution in the proof of Theorem \ref{kakati}.
%\end{rem}
%
%Finally, we establish a related theorem motivated by Remark \ref{bailando2}, where a CLO $(X, \prec_{X})$ is a linear order that satisfies \eqref{INJS}.  
\begin{thm}[$\ACAo+\QFAC^{0,1}$]\label{kakati}~ The following are equivalent. 
%$\BOOT+\QFAC^{0,1}$ follows from the following.  
\begin{itemize}
\item[(a)] $\BOOT$.
\item[(b)] A CLO $(X, \preceq_{X})$ with $X\subset \R$ is order-isomorphic to a subset of $\Q$.
\item[(c)] A CLO $(X, \preceq_{X})$ with no endpoints and $X\subset \R$ is order-isomorphic to $\Q$. 
\end{itemize}
\end{thm}
\begin{proof}
Given a CLO $(X, \preceq_{X})$ with $X\subset \R$, $\BOOT$ allows us to define the range of the associated $Y:X\di \N$ on $X$ that satisfies \eqref{INJU}.
%\be\label{INJS}
%(\forall x, y\in X)(Y(x)=Y(x)\di x=_{X} y). 
%\ee
Then use $\QFAC^{0,1}$ to pick an element from each equivalence class of the range of $Y$ on $X$.  To the resulting sequence, apply the usual `back-and-forth' argument from \cite{riot}*{p.\ 123}. 
%For the remaining direction, the proof of Theorem~\ref{wajong} is readily adapted to obtain $\BOOT$ from the second item, namely via a straightforward modification of \eqref{perfide} and the associated set $R$.  

\smallskip

For the other direction, $\BOOT$ is equivalent to $\RANGE$ as follows:
\be\label{myhunt}\tag{$\RANGE$}
(\forall G:\N^{\N}\di \N)(\exists X\subset \N)(\forall n \in \N)\big[n\in X\asa (\exists f\in \N^{\N})(G(f)=n)  ].
\ee
Indeed, the forward direction is immediate, while for the reverse direction, define $G^{2}$ as follows for $n^{0}$ and $g^{1}$: put $G(\langle n\rangle *g)=n+1$ if $Y(g, n)=0$, and $0$ otherwise. 
Let $X\subseteq \N$ be as in $\RANGE$ and note that 
\[
(\forall m \geq 1 )( m\in X \asa (\exists f\in \N^{\N})(G(f)=m)\asa (\exists g\in \N^{\N})(Y(g, m-1)=0)  ),
\]
which is as required for $\BOOT$ after trivial modification. 

\smallskip

In light of the equivalence with $\RANGE$, fix $G:[0,1]\di \N$ and note that Feferman's $\mu$ suffices to list all natural numbers $n\in \N$ with $(\exists q\in \Q\cap [0,1])(G(q)=n)$.  
%Wlog we may assume that $ A\cap \Q=\emptyset$ as Feferman's $\mu$ can enumerate $A\cap \Q$.  
Now define the set $X\subset \R$ as follows:  $y\in X$ if and only if
\[
(\exists n\in \N)(y=_{\R}n)\vee (\exists m\in \N)[ m+1< y <m+2 \wedge  G(y-(m+1))=m] .
\]
For $x, y\in X$, we put `$x=_{X} y$' if $x=_{\N} y$ or there is $m\in \N$ with 
\[
m+1<x, y<m+2\wedge G(x-(m+1))=G(y-(m+1))=m.
\]
We also put `$x\prec_{X} y$' if $x <_{\R} y  \wedge x\ne_{X} y$.  Clearly, $G$ readily yields $Y:\R\di \N$ satisfying \eqref{INJU}, i.e.\ $(X, \preceq_{X})$ is a CLO.    
Apply the second item to obtain $Q\subset \Q$ and $Z:\R\di \Q$ such that $Z$ is an order-isomorphism from $(X, \leq_{X})$ to $(Q, \leq_{\Q})$. 
%Clearly, the set $R$ is height countable and $(R, \leq_{\R})$ is a linear order.  Apply item \eqref{cloq} to obtain $Q\subset \Q$ and $Z:R\di \Q$ such that $Z$ is an order-isomorphism from $(R, \leq_{\R})$ to $(Q, \leq_{\Q})$.  
Now consider the following formula where $n\in \N$:
\begin{align}
(\exists x\in [0,1])(G(x)=n)
&\asa (\exists y\in  (n, n+1))(y\in X)\notag\\
&\asa (\exists q\in Q)(Z(n)<_{\Q}q <_{\Q}Z(n+1)).\label{perfideg}
\end{align}
The first equivalence holds by the definition of $X$, while the second equivalence follows from the fact that $Z$ is an order-isomorphism. 
Since \eqref{perfideg} is decidable given $(\exists^{2})$, we can define the range of $G$, as required for (the real-valued variation of) $\RANGE$.  
Up to coding, this yields $\RANGE$ as required.  
%Up to coding, this yields $\BOOT$ as the latter is equivalent to $\RANGE$ (Theorem \ref{rage}), which in turn expresses the existence of the range of mappings from Baire space to $\N$.
%To obtain $\QFAC^{0,1}$, fix quantifier-free $\varphi$ such that $(\forall n\in \N)(\exists f\in 2^{\N})\varphi(n, f)$.  
%Assuming some fixed coding, we define $X$ as the set of reals that code a finite sequence $w^{1^{*}}$ such that $(\forall i<|w|)\varphi(i, w(i))$.  
%We put `$x\preceq_{X}y$' if $|w|<|v|$ for the finite sequences $w, v$ coded by respectively $x, y$.  
%Then $(X, \preceq_{X})$ is a CLO and the order-isomorphism provided by the second item readily yields the required choice function for $\varphi$.  % namely similar to the proof of Theorem \ref{wajong}.
\end{proof}
On a conceptual note, the order type $\eta$ of $\Q$ appears throughout second-order RM.  However, Cantor's characterisation of $\eta$ behaves as follows:
\begin{itemize}
\item using the notion of CLO, Cantor's characterisation of $\eta$ is equivalent to $\BOOT+\QFAC^{0,1}$ (Theorem \ref{kakati}),
\item using `sets of reals that are countable' as in \eqref{INJ2}, Cantor's characterisation of $\eta$ is equivalent to $\cocode_{0}$ (\cite{dagsamXI}).
\end{itemize}
% is equivalent to $\BOOT$, which is a third-order version of $(\exists^{3})$.  
Now, $\ACAo+\BOOT$ proves $\FIVE$ while $\ACAo+\cocode_{0}$ is conservative over $\ATR_{0}$ by Remark \ref{texxchno}, i.e.\ the notion of CLO yields much stronger principles than `linear order on a countable set of reals'.
% and hence explosive by the previous and Theorem \ref{zeplo}. 

\smallskip

Next, we show that the existence of inverses in Theorem \ref{kakati} yields countable choice (Corollary \ref{kakaticor2}).  
For countable sets of reals, the inverses do not yield countable choice by Theorem \ref{kakaticor3}. 
\begin{cor}[$\ACAo$]\label{kakaticor2} The following are equivalent. 
%$\BOOT+\QFAC^{0,1}$ follows from the following.  
\begin{itemize}
\item[(a)] $\BOOT+\QFAC^{0,1}$.
\item[(b)] A CLO $(X, \preceq_{X})$ with $X\subset \R$ is order-isomorphic \emph{with an inverse} to a subset of $\Q$.
\item[(c)] A CLO $(X, \preceq_{X})$ with no endpoints and $X\subset \R$ is order-isomorphic \emph{with an inverse} to $\Q$. 
\end{itemize}
\end{cor}
\begin{proof}
That item (a) proves the other items follows from the theorem and Theorem~\ref{kakaticor3}, which is proved without reference to this corollary.  Indeed, instead of using the enumeration of $X$ as in the proof of Theorem \ref{kakaticor3}, one simply applies $\QFAC^{0,1}$ to the statement that the order isomorphism is surjective. 
For the reversal, consider \eqref{perfideg} and observe that $(\forall n\in \N)(\exists x\in [0,1])(G(f)=n)$ implies $(\forall n\in \N)(\exists q\in Q)(G(Z^{-1}(q))=n)$ where $Z^{-1}$ is the inverse of $Z$.  
Since $Q\subset \Q$, Feferman's $\mu$ yields a choice function for $(\forall n\in \N)(\exists  x\in [0,1])(G(f)=n)$.  From this, $\QFAC^{0,1}$ readily follows (via coding of reals) as required.  
\end{proof}
We recall that $\cocode_{0}$ is provable in $\Z_{2}^{\Omega}$ and hence $\ZF$, in contrast to $\QFAC^{0,1}$. 
\begin{thm}[$\ACAo$]\label{kakaticor3} The following are equivalent. 
\begin{itemize}
\item[(a)] $\cocode_{0}$.
\item[(b)] A linear order $(X, \preceq_{X})$ with countable $X\subset \R$ is order-isomorphic to a subset of $\Q$.
\item[(c)] A linear order $(X, \preceq_{X})$ with countable $X\subset \R$ is order-isomorphic \emph{with an inverse} to a subset of $\Q$.
%\item[(c)] A CLO $(X, \preceq_{X})$ with no endpoints and $X\subset \R$ is order-isomorphic \emph{with an inverse} to $\Q$. 
\end{itemize}
\end{thm}
\begin{proof}
The equivalence between the first two items is proved in \cite{dagsamXI}.  
To prove the final item, note that $\cocode_{0}$ yields an enumeration $(x_{n})_{n\in \N}$ of any countable set $X\subset \R$.  An order isomorphism $Z:X\di Q$ for $Q\subset \Q$ is surjective by definition, i.e.\ $(\forall q\in Q)(\exists x\in X)(Z(x)=q )$.  However, using the enumeration of $X$, the latter becomes $(\forall q\in Q)(\exists n\in \N)(Z(x_{n})=q )$ and the inverse $Z^{-1}:Q\di X$ can then be defined as follows using Feferman's $\mu$: $Z^{-1}(q)$ is some $x_{n}$ with $Z(x_{n})=q$.
\end{proof}
%Next, we consider item (c) in Theorem \ref{kakati} \emph{restricted} to $X\subset \R$ that are countable as a subset of reals, i.e.\ \eqref{INJ2} holds for some $Y$ and $A=X$. 
%One then readily\footnote{Note that $\cocode_{0}$ yields an enumeration $(x_{n})_{n\in \N}$ of any countable set $X\subset \R$.  An order isomorphism $Z:X\di Q$ for $Q\subset \Q$ is surjective by definition, i.e.\ $(\forall q\in Q)(\exists x\in X)(Z(x)=q )$.  However, using the enumeration of $X$, the latter becomes $(\forall q\in Q)(\exists n\in \N)(Z(x_{n})=q )$ and the inverse $Z^{-1}:Q\di X$ can then be defined as $Z^{-1}(q)$ as some $x_{n}$ with $Z(x_{n})=q$.\label{einfacherals}} defines the inverse of the order isomorphism at hand, say in $\ACAo+\cocode_{0}$.  The situation is different if we use the notion of countability as in Definition \ref{trifff} by Corollary \ref{kakaticor2}.  
%Indeed, $\QFAC^{0,1}$ is not provable in $\ZF$ while $\cocode_{0}$ is.     
%
%
%\noindent
%The following exercise is an interesting variation of Theorems \ref{wajong} and \ref{kakati}.
%\begin{rem}[$\ACAo+\IND_{2}$]\rm ~
%\begin{itemize}
%\item Show that $\enum$ is equivalent to the generalisations of items (a) and (b) from Theorem \ref{wajong} that state the existence of inverses for the order morphisms at hand. 
%\item Show that the existence of an inverse for the order morphism in the second item of Theorem \ref{kakati} yields countable choice as in $\QFAC^{0,1}$.  
%\end{itemize}
%\end{rem}
%
%\smallskip
%
Next, we also briefly study well-orderings as in the following (usual) definition.  
\bdefi[Well-orderings]~
\begin{itemize}
\item A \emph{well-ordering} $(X, \prec_{X})$ is a well-founded linear order, i.e.\ it has no strictly descending sequences. 
\item A well-ordering $(X, \preceq_{X})$ with $X\subset \N^\N$ is a \emph{cwo} in case there exists $Y:\N^{\N}\di \N$ satisfying \eqref{INJU}.
\end{itemize}
\edefi   
We have the following result on the comparability of well-orderings, a big topic in second-order RM.   The second-order version of the centred statement, called $\CWO$, is equivalent to $\ATR_{0}$ by \cite{simpson2}*{V.6.8}. 
\begin{thm}[$\ACAo+\QFAC^{0,1}+\SIND$]\rm
The principle $\BOOT$ is equivalent to:
\begin{center}
\emph{For cwos $(X, \preceq_{X}) $ and $(Y, \preceq_{Y})$ with $X, Y\subset \R$, the former order is order-isomorphic to the latter or an initial segment of the latter, or vice versa.}
\end{center}
%Modify the proof of Theorem \ref{CWOOO} if necessary.  
\end{thm}
\begin{proof}
To derive the centred statement from $\BOOT$, use the latter plus $\QFAC^{0,1}$ to obtain enumerations of the cwos involved. 
Clearly, $\ACAo+\BOOT$ proves $\FIVE$ and the latter implies $\CWO$ by \cite{simpson2}*{V.6.8}.  % while Remark \ref{LEM} guarantees $(\exists^{2})$ is available. 
Hence, we can apply $\CWO$ to the aforementioned enumerations and obtain the centred statement. 
%The proof of $\enum\di \IND_{0}$ is straightforward.  

\smallskip

For the other direction, we shall prove $\RANGE$ from the proof of Theorem \ref{kakati}, i.e.\ fix some $G:\N^{\N}\di \N$.  
In case $(\exists m\in\N)(\forall f\in \N^{\N})(G(f)\leq m)$, $\SIND$ readily yields the range of $G$.
% as we have 
%\[
%(\forall n\in \N)(\exists \textup{ at most one } x\in [0,1])(x\in A\wedge Y(x)=n).
%\]    
Hence, we may assume $(\forall m\in \N)(\exists f\in \N^{\N})(G(f)> m)$.  
Now define the linear order $(X, \preceq_{X})$ where $X=\N^{\N}$ and
\[
x\preceq_{X} y \equiv \big[G(y)=n_{0} \vee [  G(x)\ne n_{0}\wedge G(x)\leq_{\N}G(y)   ]\big],
\]
where $n_{0}\in \N$ is the least $n\in \N$ such that $(\exists f\in \N^{\N})(G(f)=n)$; this number is readily defined using $\SIND$. 
Let $y_{0}\in X$ be such that $G(y_{0})=n_{0}$
The functional $G$ is such that $G(x)=G(y)$ implies $x=_{X}y$ for any $x, y\in X$, i.e.\ $(X, \preceq_{X})$ is a cwo.  
%Let $y_{0}, \dots, y_{k}\in A$ be all $y\in A$ with $H(y)=n_{0}$, in ascending order.
Intuitively, $(X,\preceq_{X})$ has order type $\omega+1$, i.e.\ the order of $\N$ followed by $1$ element.  
Hence, of the different possibilities provided by the centred statement, all-but-one to contradiction.  Indeed, a finite initial segment of either $(\N, \leq_{\N})$ or $(X, \preceq_{X})$ has only got finitely many elements, while $\N$ is infinite and $X$ satisfies $(\forall m\in \N)(\exists x\in X)(G(x)\geq m)$.  Similarly, an order-isomorphism $W:X\di \N$ leads to contradiction as follows: since there is $y_{0}\in X$ such that $W(y_{0})=n_{0}$, there cannot be a injection from $X\setminus \{y_{0}\}$ to $\{0, 1,\dots, W(y_{0})\}$, as the latter set is finite, while the former is not. 
Similarly, an order-isomorphism $Z:\N\di W$ yields a contradiction as any $n\geq n_{0}$ is mapped below $Z(n_{0})\in X$ (relative to $\preceq_{X}$), which is not possible.  
The only remaining possibility is an order-isomorphism $Z: \N\di X\setminus \{y_{0}\}$, where the latter is the initial segment $\{ y\in X: y \prec_{X} y_{0} \}$.  
The mapping $Z$ allows us to define the range of $G$, as required, and we are done.  
\end{proof}
%By Theorem \ref{wajong}, the second item of Theorem \ref{kakati} restricted to countable sets of reals is (only) equivalent to $\enum$, while one obtains $\BOOT$ for CLOs.  
In light of the above results, the notion of countability based on \eqref{INJU} seems (much) more general than `countable set of reals': the latter gives rise to equivalences involving $\cocode_{0}$, pioneered in \cite{dagsamXI}, while the former readily yields $\BOOT$ and countable choice $\QFAC^{0,1}$ (Theorem \ref{kakati}).  
%We provide another example as follows. 

\smallskip

\subsection{Countable metric spaces}\label{countri}
We study the RM-properties of metric spaces that are countable in the sense of Definition \ref{trifff}.    In particular, we study the following theorem by Sierpi\'nski (see \cite{dalzielE} and \cite{kura}*{p.\ 287}).  % based on Definition \ref{trifff}.  
\begin{center}
\emph{A countable metric space without isolated points is homeomorphic to the rationals.}
\end{center}
We obtain an equivalence for $\BOOT+\QFAC^{0,1}$ in Theorem \ref{kabukischlabiku}.
We stress that Friedman studies similar comparability theorems for countable metric spaces in \cite{friedmet}*{\S2} via codes, with most results provable in $\ACA_{0}$ or $\ATR_{0}$.  

\smallskip

First of all, the following definition is standard.  
\bdefi
Two metric spaces $(X_{0}, d_{0})$ and $(X_{1}, d_{1})$ are \emph{homeomorphic} if there are $F:X_{0}\di X_{1}$ and $F^{-1}: X_{1}\di X_{0}$ that are injective, surjective, continuous, and satisfy
 $F^{-1}(F(x_{0}))=_{X_{0}}x_{0}$ for $x_{0}\in X$ and $F(F^{-1}(x_{1}))=_{X_{1}}x_{1}$ for $x_{1}\in X_{1}$.  We refer to $F^{-1}$ as `the inverse of $F$'.
\edefi
We now have the following equivalence for Sierpi\'nski's classification theorem. 
\begin{thm}[$\ACAo$]\label{kabukischlabiku}~
The following are equivalent.
\begin{itemize}
\item The combination $\BOOT+\QFAC^{0,1}$.
\item Any countable metric space $(M, d)$ with $M\subset \N^{\N}$ and without isolated points, is homeomorphic to the rationals.  
\end{itemize}
\end{thm}
\begin{proof}
For the downward implication, consider a metric space $(M, d)$ with $M\subset \N^{\N}$ with $Y:M\di \N$ injective.  
Use $\BOOT$ to define the range of $Y$, i.e.\ $X=\{n\in \N: (\exists x\in M)(Y(x)=n )\}$.  
Apply $\QFAC^{0,1}$ to $(\forall n\in \N)(\exists x\in M)( n\in X\di Y(x)=n)$ to obtain an enumeration of $M$.
Given this enumeration, the elementary proof of the second item in \cite{dalzielE} goes through verbatim in $\ACAo$. 

\smallskip

For the upward implication, we first obtain $\RANGE$, i.e.\ let $G^{2}$ be fixed.  
Let $(q_{n})_{n\in \N}$ be an enumeration of $\Q$, put $M=\N^{\N}$, and define $Z:M\di \R$ as follows:
\[
Z(f):=
\begin{cases}
q_{n} &  f(0)=n \textup{ is odd }\\
\pi\cdot q_{G(f(1)*f(2)*\dots)} & f(0)=n \textup{ is even }
\end{cases}.
\]
We also put $f=_{M}g$ if and only if $Z(f)=Z(g)$ and define $d(f, g):=|Z(f)-Z(g)|$.  
Then $(M, d)$ is a metric space which is moreover countable as $Z$ readily provides an injection to $\N$ (relative to $=_{M}$).  
Essentially by the definition of the metric (and the density of $\Q$ in $\R$), there are no isolated points in $(\N^{\N}, d)$.  
Then $(\N^{\N}, d)$ and $(\Q, |\cdot |_{\Q})$ are homeomorphic and suppose $F, F^{-1}$ are the associated homeomorphism and its inverse.  
By definition, we have 
\be\label{zongbird}
(\exists f\in \N^{\N})(G(f)=n)\asa (\exists q\in \Q)( G( F^{-1}(q)(1)*F^{-1}(q)(2)*\dots   )=n  ),
\ee
which yields $\BOOT$ as the right-hand side of \eqref{zongbird} is decidable given $(\exists^{2})$.  
Regarding $\QFAC^{0,1}$, observe that if $(\forall n\in \N)(\exists f\in \N^{\N})(G(f)=n)$, then a choice function is obtained from the sequence $(F^{-1}(q_{n})(1)*F^{-1}(q_{n})(2)*\dots)_{n\in \N}$.  
This readily yields $\QFAC^{0, 1}$, and we are done.  
\end{proof}
In line with Definition \ref{trifff}, we have the following stronger notion.  
\bdefi
A set $X$ with equivalence relation $=_{X}$ is \emph{strongly countable} if there is $Y:X\di \N$ that satifies \eqref{edelkraft} and $(\forall n\in \N)(\exists x\in X)( Y(x)=n )$.  
\edefi
\begin{cor}[$\ACAo$]\label{kabukischlabikucor}~
The following are equivalent.
\begin{itemize}
\item The axiom $\QFAC^{0,1}$.
\item Any strongly countable metric space $(M, d)$ with $M\subset \N^{\N}$ and without isolated points, is homeomorphic to the rationals.  
\end{itemize}
\end{cor}
\begin{proof}
Consider the proof of the theorem.  
For the downward implication, note that $\BOOT$ is no longer needed.  
The upward implication follows as in the proof of Corollary \ref{kakaticor2}, based on \eqref{zongbird}.  
\end{proof}
Next, we consider the following related statement, studied in \cite{earlyhirst3}.
\begin{center}
\emph{Let $X, Y$ be countable, closed, and totally bounded subsets of a complete metric space $(M, d)$.  Then either $X$ is homeomorphic to a subset of $Y$ or vice versa.}
\end{center}
The reader is invited to show that this statement is equivalent to $\BOOT+\QFAC^{0,1}$.  
Formulated using codes, this statement is equivalent to $\ATR_{0}$ by \cite{earlyhirst3}*{Theorem 4.1}.  

\smallskip

In conclusion, we have obtained $\BOOT$ and $\QFAC^{0,1}$ from basic statements about countable metric spaces.  
Weaker statements about strongly countable sets turn out to be equivalent to $\QFAC^{0,1}$.

\subsection{Countable second-countable spaces}\label{SPAST}
\subsubsection{Introduction}
We study RM-properties of second-countable spaces that are countable in the sense of Definition \ref{trifff}.  
In particular, we show that $\BOOT$ is equivalent to the associated supremum principle for continuous functions (Section~\ref{zenfff}).  Scattered spaces and metrisability are also briefly discussed.  
We also show that the Ginsburg-Sands theorem for countable spaces as in Definition \ref{trifff} is equivalent to countable choice as in $\QFAC^{0,1}$ (Section \ref{kamirpamir}).
The former theorem has been studied recently in second-order RM (\cite{benham}).  

\smallskip

First of all, for the rest of this section, we tacitly assume that $X$ is an RSC space with $X\subset \R$, basis $(U_{i})_{i\in \N}$, and mapping $k$ as in Definition \ref{CSC}.  
The associated notion of topological (in)distinguishability is defined as follows.  
\bdefi
For $x, y, \in X$, we write `$x\equiv_{X} y$' in case $(\forall i \in \N)(x\in U_{i}\asa y\in U_{i})$.  We say that `$x$ and $y$ are topologically indistinguishable'.
\edefi
Secondly, we have previously defined an RSC space $X$ to be countable if the set $X$ has this property, namely in \cite{samSECOND}.  
While this choice seems natural and yields nice equivalences involving $\cocode_{0}$, the following definition turns out to be more general.  
We motivate the use of `$\equiv_{X}$' in Remark \ref{kokojambbo} beyond the observation that it naturally comes to the fore for the reals and metric spaces. 
\bdefi\label{hyyinx}
A CRSC-space is an RSC space $X$ for which there exists $Y:\R\di \N$ such that 
\be\label{INJX}
(\forall x, y\in X)( Y(x)=_{\N} Y(y)\di x\equiv _{X} y).  
\ee
\edefi
Finally, we motivate the use of topological indistinguishability in \eqref{INJX} based on properties of separation axioms of topological spaces.  
\begin{rem}[Separation axioms]\rm\label{kokojambbo}
For topological spaces, there are a number of equivalent formulations of the $T_{1}$-axiom.   One of these is the condition \emph{finite sets are closed} by \cite{theark}*{p.\ 28, Prop.\ 13} or \cite{zot}*{\S 16.7}.  
Clearly, the space from the proof of Theorem \ref{fisterz} satisfies the latter condition.   Of course, $T_{1}$-spaces are $T_{0}$ in general, i.e.\ distinct points are topologically distinguishable.  
Hence, the use of topological indistinguishability in \eqref{INJX} is natural enough.  
%Thus, our definition of CRSC-space is as general as possible, though based on an equivalent definition.  
Moreover, the space from the proof of Theorem \ref{zupb2} is even perfectly normal.  % i.e.\ $T_{6}$.    
\end{rem}

\subsubsection{Supremum principle}\label{zenfff}
In this section, we shall establish an equivalence between $\BOOT$ and the supremum principle for compact CRSC-spaces, necessitating the following definitions. 
\bdefi
An RSC-space $X$ is \(open cover\) compact if for any covering\footnote{A mapping $\Psi:X\di \N$ generates a covering of $X$ in case $x\in U_{\Psi(x)}$ for any $x\in X$.  
Both Cousin and Lindel\"of (only) study this kind of uncountable coverings in \cites{cousin1, blindeloef}, which are the papers in which the Cousin and Lindel\"of lemmas first appeared.\label{eber}} generated by $\Psi:X\di \N$, there are $x_{0}, \dots, x_{k}\in X$ such that $\cup_{i\leq k}U_{\Psi(x_{i})}$ covers $X$, 
\edefi
\bdefi
For an RSC-space $X$, a function $f:X\di \R$ is continuous if $f^{-1}(V)$ is open in $X$ for every open $V\subset \R$.  
\edefi
%We recall the following principle from Section \ref{dacount}.
\begin{princ}[$\SUP$]\label{zupb2}
Let $X$ be a compact RSC-space.
For a continuous function $f:X\di \R$ and a decreasing sequence $(C_{n})_{n\in \N}$ of closed sets, there is a sequence $(x_{n})_{n\in \N}$ such that $x_{n}=\sup_{x\in C_{n}}f(x)$.
\end{princ}
The principle $\SUP$ restricted to countable $X\subset \R$ is provable in $\ACAo+\cocode_{0}$ using the obvious proof.  
By contrast, we have the following equivalence. 
\begin{thm}[$\ACAo+\QFAC^{0,1}$]\label{fisterz} The following are equivalent.
\begin{itemize}
\item $\BOOT$
\item The supremum principle $\SUP$ for CRSC-spaces.%\label{1itemd}
%\item The previous item restricted to effectively or strongly continuous functions.\label{1iteme}
%\item Item \eqref{1iteme} restricted to sequences $(C_{n})_{n\in \N}$ of \emph{uniformly} closed sets.\label{1itemf}
%%\item The Ginsburg-Sands theorem $\GS$ for countable RSC-spaces.\label{1gsitem}
%\item The enumeration principle $\cocode_{0}$.
\end{itemize}
%The equivalence holds for `sequential', `countable', and `limit point' compactness, the latter additionally assuming $\QFAC^{0,1}$.
\end{thm}
\begin{proof}
The downward direction follows as in the proof of Theorem \ref{kakati}.
Indeed, $\BOOT+\QFAC^{0,1}$ yields an enumeration of any CRSC-space, after which the proof is straightforward. 

\smallskip

To obtain $\BOOT$ from the second item, fix $G:[0,1]\di \N$ and note that we may assume that $G(x)\ne G(0)$ for $x\in (0,1]$.
Now define the sequence $(U_{i})_{i\in \N}$ as follows:
\be\label{1terfz}
U_{2n}=\{x\in (0,1]: G(x)=n \} \textup{ and }  U_{2n+1}=\{  0 \}\cup \{ x\in (0,1]:G(x)> n \}.
\ee
Use $(\exists^{2})$ to define $k:([0,1]\times \N^{2})\di \N$ as follows for $x\in [0,1]$ and $i\leq j$ in $\N$
\be\label{kkkkz}
k(x, i, j):=
\begin{cases}
%i  & \textup{ if $i$ and $j$ are even,}\\
j  & \textup{ if $i$ and $j$ are odd,}\\
j  & \textup{ if $i$ is odd and $j$ is even,}\\
i  & \textup{ otherwise,}\\
\end{cases}
\ee
and observe that this mapping has the required properties for forming a base of $X=[0,1]$. 
Hence, $X$ is an RSC-space, which can be seen to be compact as follows:  let $\Psi:X\di \N$ be given and note that $n_{0}=\Psi(0)$ must be odd to guarantee $0\in U_{n_{0}}$ following \eqref{1terfz}. 
Clearly, $X$ is covered by $\cup_{i\leq  n_{0}}U_{2i}\cup U_{n_{0}}$.  %In exactly the same way, one proves that $X$ is countably-compact.  
Moreover, $X$ is also a CRSC-space, as $x\equiv_{X} y\asa G(x)=G(y)$ for $x, y\in X\setminus \{0\}$.  Define $f:X\di \R$ by $f(x):=\frac{1}{2^{G(x)}}$ if $x\ne 0$ and $f(0):=0$.  
One readily proves that $f$ is continuous.  Now apply $\SUP$ to the sequence $C_{n}:= X\setminus \cup_{i< n}U_{2i}$ and observe that for $n\geq 1$ we have:
\be\label{tochuniz}\textstyle
(\exists x\in [0,1])(G(x)=n)\asa [\frac{1}{2^{n}}= \sup_{x\in C_{n}}f(x)].
\ee
Since the right-hand side of \eqref{tochuniz} is decidable using $(\exists^{2})$, we can define the range of $G$.  
Up to coding, this yields $\RANGE$ and we are done.  
\end{proof}
The reader can verify that the maximum principle, where the supremum in $\SUP$ is attained, is equivalent to $\BOOT+\QFAC^{0,1}$. 

\smallskip

%\begin{rem}[$\FIVE^{\omega}$]\rm
%Show that $\SUP$ restricted to countable sets $X\subset \R$ implies $\SIX$.  Compare to the theorem.   
%\end{rem}
Next, a space is \emph{scattered} if every non-empty subset has an isolated point.  The centred statement is proved by Hausdorff in \cite{huisdorp}*{Ch.\ VIII, \S3}.    
\begin{cor}[$\ACAo+\QFAC^{0,1}$] The following statement implies $\BOOT$.
\begin{center}
\emph{a scattered RSC-space $X$ can be enumerated.}
\end{center}
Moreover, the equivalence in Theorem~\ref{zupb2} goes through for `CRSC-space' replaced by `scattered CRSC-space'.  
\end{cor}
\begin{proof}
The RSC space $X$ defined using \eqref{1terfz} is clearly scattered due to the presence of $U_{2n}$.  
This establishes both the equivalence and the fact that $\BOOT$ follows from the centred statement. 
%Under what conditions is a reversal possible?  Show that the equivalence in Theorem~\ref{zupb2} goes through for `CRSC-space' replaced by `scattered RSC-space'.  
\end{proof}
Hausdorff's proof from \cite{huisdorp} of the centred statement goes through in $\ACAo+\BOOT+\QFAC^{0,1}$, using the fact that the union of enumerable sets is enumerable.  

\smallskip

Regarding further results, the \emph{Urysohn metrisation theorem} states that a regular Hausdorff and second-countable space is metrisable, i.e.\ homeomorphic to some metric space.  
This theorem is studied in second-order RM (\cite{damurm}*{\S10.8.3}) for homeomorphisms to complete and separable metric spaces.  
One readily shows that $\BOOT$ follows from the statement that a regular Hausdorff CRSC space is homeomorphic to a separable metric space. 
%\item Assuming $(\SS^{2})$, show that $\SIX$ follows from the statement that a regular Hausdorff RSC space $X$ for countable $X\subset \R$ is homeomorphic to a separable metric space.  
%\end{itemize}
%What is the difference between the proofs of these two items?
%\end{rem}

\smallskip

%\begin{rem}\rm
%Which of the above theorems concerning second-countable spaces go through for the restrictions discussed in Remark \ref{kokojambbo}?
%\end{rem}
In conclusion, countability as in \eqref{INJX} gives rise to theorems that reverse to $\BOOT$, which is more general than `real countability' as in \eqref{INJ2}.  % the latter notion `only' gives rise to theorems equivalent to $\enum$.   

\smallskip

\subsubsection{The Ginsburg-Sands theorem}\label{kamirpamir}
We show that countable choice as in $\QFAC^{0,1}$ is equivalent to the Ginsburg-Sands theorem for second-countable spaces that are countable in the sense of Definition \ref{hyyinx}.  
The second-order versions of the Ginsburg-Sands theorem are all provable in $\ACA_{0}$ (\cite{benham}) while the higher-order versions based on countable sets of reals are provable in $\ACAo+\cocode_{0}$ (\cite{samBOOK, samSECOND}).  

\smallskip

First of all, the Ginsburg-Sands theorem is formulated as follows in \cite{zieginds}*{p.\ 574}.   
%First of all, we study the Ginsburg-Sands theorem  for RSC-spaces as in Principle \ref{GS}.  
%The former has been studied in RM (\cite{benham}) for CSC-spaces, i.e.\ RCS-spaces that come with an enumeration.  
\begin{princ}[$\GS$]\label{GS}
An infinite topological space has a sub-space homeomorphic to exactly one of the following topologies over $\N$:
\begin{itemize}
\item The discrete topology: all sets are open.
\item The indiscrete topology: only $\emptyset$ and $\N$ are open.  
\item The co-finite topology: the open sets are $\emptyset$, $\N$, and any subset of $\N$ with finite complement. 
\item The initial segment topology: the open sets are $\emptyset$, $\N$, and any set of the form $[0, n]=\{k\in \N : k\leq n \}$.
\item The final segment topology: the open sets are $\emptyset$, $\N$, and any set of the form $[n, +\infty)=\{k\in \N : n\leq k \}$.
\end{itemize}
\end{princ}
%The results on $\GS$ and related topics may be found in Section \ref{boco} and establish the observation (O1).  
%We note that the second-order version of the Ginsburg-Sands theorem for CSC-spaces is provable in $\ACA_{0}$ by \cite{benham}*{\S4}.
%The Ginsburg-Sands theorem will be seen to have interesting properties in both second- and higher-order RM.  
%
%\smallskip
Secondly, we obtain an equivalence involving $\GS$ and countable choice.    
\begin{thm}[$\ACAo+\IND_{3}$]\label{true} 
The following are equivalent. 
\begin{itemize}
\item The Ginsburg-Sands theorem\footnote{A set $X$ is called `infinite' if it contains arbitrary many pairwise distinct elements.  
A more detailed discussion of finiteness can be found in Section \ref{pifpoefpaf}.  Instead of infiniteness, we could require the existence of arbitrary many pairwise topologically distinguishable elements; the equivalence in Theorem \ref{true} would still go through.  \label{foetsie}} $\GS$ for CRSC-spaces.
\item Countable choice as in $\QFAC^{0,1}$.
\end{itemize}
\end{thm}
\begin{proof}
First of all, we show that $\GS$ for CRSC-spaces implies $\QFAC^{0,1}$.  To this end, let $\varphi$ be quantifier-free and such that $(\forall n\in \N)(\exists x\in \R)\varphi(n, x)$. 
We assume a fixed coding of finite sequences of reals as real numbers, which is readily defined in $\ACAo$. 
Now define $w^{1^{*}}\in X$ if and only if $(\forall i<|w|)\varphi(n, w(i))$, i.e.\ $X$ contains initial segments of the choice function we are after.  
Define $U_{n}:= \{ w\in X : |w|=n  \}$ and note that we obtain a base in the obvious way.  
Moreover, we have that 
\[
|w|=|v|\asa (\forall n\in \N)(v\in U_{n}\asa w\in U_{n})\asa w\equiv_{X}v,
\]
which shows that $X$ is an CRSC-space.  
Thanks to $\IND_{3}$, $X$ contains arbitrary many pairwise topologically distinct elements. 
Clearly, all sets are open in $X$ and the same for any sub-space of $X$.  
By $\GS$ as in the first item, $X$ has a sub-space that is homeomorphic to the discrete topology on $\N$.  
Let $f:X\di \N$ and $f^{-1}:\N\di X$ be the associated continuous bijection and its inverse.  Now consider the sequence $(f^{-1}(n))_{n\in \N}$ in $X$.  
By definition, $f^{-1}(n)$ is a finite sequence and since $f:X\di \N$ is a bijection, we have that $(\forall m\in \N)(\exists n\in \N)(\exists x\in f^{-1}(n))\varphi(m, x)$.  
Thus, $(f^{-1}(n))_{n\in \N}$ readily yields a choice function for $(\forall n\in \N)(\exists x\in \R)\varphi(n, x)$.
Since $\ACAo$ is available, $\QFAC^{0,1}$ also follows.  

\smallskip

Secondly, to prove $\GS$ for an CRSC-space $X$, apply $\QFAC^{0,1}$ to the statement that $X$ is infinite$^{\ref{foetsie}}$, which readily yields a sequence of elements of $X$.  
Identifying an element with its index, apply the second-order $\GS$, provable in $\ACA_{0}$ by \cite{benham}*{Theorem~4.6}, to obtain the required result.
\end{proof}
In conclusion, countability as in \eqref{INJX} gives rise to versions of $\GS$ that reverse to $\QFAC^{0,1}$, which is more general than `real countability' as in \eqref{INJ2}.  % the latter notion `only' gives rise to theorems equivalent to $\enum$.   

\subsection{Finite sets and lemmas by K\"onig}\label{pifpoefpaf}
We discuss the notion of finite set based on the above observations regarding countability and equivalence relations (Section \ref{othernamez3}).  
Armed with this knowledge, we show that K\"onig's infinity lemma is equivalent to $\QFAC^{0,1}$ and other principles (Section \ref{dachion}).  

\subsubsection{Finiteness by any other name}\label{othernamez3}
We recall the definition of `finite set of reals' and observe the essential nature of the associated equivalence relation.  
%Since the focus of this paper is on countable sets, we shall be brief.  

\smallskip

First of all, the following definition of finite set of reals is most suitable for the development of higher-order RM.  
This claim is substantiated by the many equivalences in \cites{dagsamXI, samBIG, samBIG2, samBIG3} based on this definition, e.g.\ in Fourier analysis. 
\bdefi\label{koenkruk}
A set $A\subset \R$ is finite if there is $N\in \N$ such that if $x_{0}, \dots, x_{N}\in A$ are pairwise different, then there is $j\leq N$ with $x_{j}\not\in A$.
\edefi
We refer to the number $N\in \N$ from Definition \ref{koenkruk} as a `size bound' of $A$ and write $|A|\leq N$.  
Note that `pairwise different' means $x_{i}=_{\R}x_{j}\di i=_{\N}j$ for $i,j\leq N$, i.e.\ the equivalence relation $=_{\R}$ is again essential.
Indeed, each real number has uncountably many different representations that are equal under $=_{\R}$.  
Hence, a set $A\subset \R$ satisfying Definition \ref{koenkruk} is therefore finite \emph{from the point of view of $\R$}, but does contain uncountably many different elements of Baire space, divided among finitely many (at most $N$) equivalence classes.  
%Hence, a finite set as in Definition \ref{koenkruk} 
%therefore contains uncountably different elements of Baire space

\smallskip

Secondly, let $(X, =_{X}) $ be a pair consisting of a set $X\subset \N^\N$ and an equivalence relation `$=_{X}$'.  
The associated general notion of `finiteness' is then of course defined relative to $=_{X}$, just like in Definition \ref{koenkruk}.  
\bdefi[Finite sets]\label{trukkefoor}
A set $A$ in $(X, =_{X})$ is \emph{finite} if there is $N\in \N$ such that if $x_{0}, \dots, x_{N}\in A$ are pairwise different, i.e.\ $x_{i}=_{X}x_{j}\di i=_{\N}j$ for $i,j\leq N$, then there is $j\leq N$ with $x_{j}\not\in A$.
\edefi
In conclusion, the definition of `finite subset of $X$' crucially hinges on the associated equivalence relation, namely $=_{X}$.

\subsubsection{On lemmas by K\"onig}\label{dachion}
We study the RM-properties of eponymous lemmas by K\"onig from \cite{koning147,koning26} motivated by the observations in Section \ref{othernamez3}.   
In particular, we show that the \emph{K\"onig's infinity lemma} is equivalent to countable choice as in $\QFAC^{0,1}$.  By contrast the second-order version of this lemma is provable in $\ACA_{0}$ while the restriction of this lemma to countable sets in Baire space can be proved in $\ACAo+\cocode_{0}$ plus a fragment of the induction axiom (\cite{samBOOK}).  

\smallskip

First of all, we discuss the history of K\"onig's lemmas as the literature contains some errors.  
%\begin{rem}[K\"onig's K\"onig's lemma]\rm
We let \emph{K\"onig's tree lemma} be the statement \emph{every infinite finitely branching tree has a path}.  When formulated in the language of second-order arithmetic, the tree lemma is equivalent to $\ACA_{0}$ by \cite{simpson2}*{III.7.2}.  
The paper \cite{koning147} is cited as the original source for K\"onig's tree lemma in \cite{simpson2}*{p.\ 125}, but \cite{koning147} does not mention the word `tree' (i.e.\ the word `Baum' in German).  In fact, K\"onig's original (graph theoretic) lemma from \cite{koning147} is as follows, translated from German.
\begin{princ}[K\"onig's infinity lemma for graphs]\label{KILq}
{If a countably infinite graph $G$ can be written as countably many non-empty finite sets $E_{1}, E_{2}, \dots$ such that each point in $E_{n+1}$ is connected to a point in $E_{n}$ via an edge, then $G$ has an infinite path $a_{1}a_{2}\dots$ such that $a_{n}\in E_{n}$ for all $n\in \N$. }
%Zerf\UTF{00E4}llt die abz\UTF{00E4}hibarunendliche Menge der Punkte (= Knoten- punkte) eines unendlichen Graphen G in abz\UTF{00E4}hlbar viele endliche
%nicht leere Mengen Eu E2, E3,... derart, dass jeder Punkt von En+l (n = 1, 2,3,...) mit einem Punkte von E\UTF{201E} durch eine Kante ver- bunden ist, so gibt es im Graphen einen \UTF{201E}unendlichen Weg" atafas..., der aus jeder der Mengen E\UTF{201E} einen Punkt a\UTF{201E} enth\UTF{00E4}lt.
\end{princ}
The original version, introduced a year earlier in \cite{koning26}, is formulated in the language of set theory as follows, in both \cite{koning26, koning 147}.  
\begin{princ}[K\"onig's infinity lemma for sets]\label{KIL2q}
Given a sequence $E_{0}, E_1, \dots$ of finite non-empty sets and a binary relation $R$ such that for any $x\in E_{n+1}$, there is at least one $y\in E_{n}$ such that $yRx$.
Then there is an infinite sequence $(x_{n})_{n\in \N}$ such that for all $n\in \N$, $x_{n}\in E_{n}$ and $x_{n}Rx_{n+1}$.
% Es sei EitE2,E3 ... eine abz\UTF{00E4}hibarunendliche Folge end-
%licher, nicht leerer Mengen und, R eine bin\UTF{00E4}re Relation, die so beschaffen ist, dass zu jedem Element x\UTF{201E}+i von E\UTF{201E}+1 mindestens ein solches Element x\UTF{201E} von E\UTF{201E} geh\UTF{00F6}rt, welches zu xn+1 in der Relation R steht, was wir durch x\UTF{201E}Rxn+l ausdr\UTF{00FC}cken wollen. Dann kann man in jeder der Mengen E\UTF{201E} je ein Element a\UTF{201E} derart bestimmen, dass f\UTF{00FC}r die unendliche Folge a,. a2, a3,... stets a\UTF{201E} R a\UTF{201E}+x bestehe
\end{princ}
The names \emph{K\"onig's infinity lemma} and \emph{K\"onig's tree lemma} are used in \cite{wever} which contains a historical account of these lemmas, as well 
as the observation that they are equivalent; the formulation involving trees apparently goes back to Beth around 1955 in \cite{bethweter}, as also discussed in detail in \cite{wever}.
When formulated in set theory, K\"onig's infinity lemma is equivalent to a fragment of the Axiom of Choice (\cite{levy1}*{p.\ 298}) over $\ZF$, which is mirrored by Theorem \ref{wood17}.    % and is (strictly) implied by Ramsey's theorem (\cite{fortru}).
%The following quote by K\"onig constitutes motivation and evidence that graph theory was intended to be infinitary.      
%\begin{quote}
%Diese Bemerkung ist wichtig, da, wenn man sie einmal angenommen hat, nichts im Wege steht die ``Sprache der Graphen'' auch dann zu nuetzen, wenn die Mengen [\dots] nicht endlich, ja sogar von beliebig grosser Machtigkeit sind. (\cite{koning16}*{p.\ 460})
%\end{quote}
%The final sentence states that graphs of any cardinality can be studied in graph theory.  
%Moreover, we note that K\"onig's infinity lemma is introduced in \cite{koning147} as a graph-theoretic formulation of another theorem from \cite{koning26}.  In both the French (\cite{koning26}*{\S3}) and German formulation (\cite{koning147}*{\S1}), the word `sequence' is used in the conclusion, i.e.\ an infinite path is a sequence of elements.  By contrast, the antecedent is always formulated using countable sets.  
%\end{rem}
%

\smallskip
%
%First of all, K\"onig's infinity lemma $\Korg_{2}$ as in Principle \ref{KIL2real} is clearly provable in $\ACA+\enum$ as the set $\cup_{n\in \N}E_{n}\subset \R$ therein is height-countable and can be enumerated.  
%By contrast, the generalisation $\Korg_{3}$ in Principle \ref{KIL22real} implies $\QFAC^{0,1}$ by Theorem \ref{wood17}.  

\smallskip

Secondly, we study following version of K\"onig's infinity lemma.
The only modification compared to Principle \ref{KIL2q} is the restriction to $(X, =_{X})$ where $X\subset \N^{\N}$.  
%Recall that a binary relation $R$ is given by a characteristic function $F_{R}:\R^{2}\di \R$, i.e.\ $xRy\equiv F_{R}(x, y)\geq_{\R} 0$. 
\begin{princ}[$\KINL$]\label{KIL22real}
Let $(E_{n})_{n\in \N}$ be a sequence of sets in $(X, =_{X})$ and let $R$ a binary relation such that for all $n\in\N$ we have:
\begin{itemize}
\item the set $E_{n}$ is finite and non-empty,
\item for any $x\in E_{n+1}$, there is at least one $y\in E_{n}$ such that $yRx$.
\end{itemize}
% finite non-empty sets and a binary relation $R$ such that for any $x\in E_{n+1}$, there is at least one $y\in E_{n}$ such that $yRx$.
Then there is a sequence $(x_{n})_{n\in \N}$ such that for all $n\in \N$, $x_{n}\in E_{n}$ and $x_{n}Rx_{n+1}$.
% Es sei EitE2,E3 ... eine abz\UTF{00E4}hibarunendliche Folge end-
%licher, nicht leerer Mengen und, R eine bin\UTF{00E4}re Relation, die so beschaffen ist, dass zu jedem Element x\UTF{201E}+i von E\UTF{201E}+1 mindestens ein solches Element x\UTF{201E} von E\UTF{201E} geh\UTF{00F6}rt, welches zu xn+1 in der Relation R steht, was wir durch x\UTF{201E}Rxn+l ausdr\UTF{00FC}cken wollen. Dann kann man in jeder der Mengen E\UTF{201E} je ein Element a\UTF{201E} derart bestimmen, dass f\UTF{00FC}r die unendliche Folge a,. a2, a3,... stets a\UTF{201E} R a\UTF{201E}+x bestehe
\end{princ}
Thirdly, the Axiom of Choice for finite sets is also studied in RM (\cite{gohzeg, skore3}).  We study the following instance involving again $(X, =_{X})$ where $X\subset \N^{\N}$.  
\begin{princ}[Axiom of finite Choice]\label{gohkuku}
Let $(X_{n})_{n\in \N}$ be a sequence of finite non-empty sets in $(X, =_{X})$.  Then there is $(x_{n})_{n\in \N}$ with $x_{n}\in X_{n}$ for all $n\in \N$.
\end{princ}
%It implies $\QFAC^{0,1}$ and hopefully $\QFDC^{1,1}$.
%and show that they inhabit the range of hyperarithmetical analysis when formulated close(r) to their original formulation.  
%A study of this topic from the point of view of a more general notion of countability can be found in Section~\ref{pifpoefpaf}.  
Fourth, we have Theorem \ref{wood17} where $\IND_{3}$ was introduced in Section \ref{prelim}.
% is equivalent to the fragment of the induction axiom for formulas $\varphi(n)\equiv(\forall f\in \N^{\N})(\exists g\in \N^{\N})\psi(f, g, n)$ with $\psi$ arithmetical. 
\begin{thm}[$\ACAo+\IND_{3}$]\label{wood17} The following are equivalent. 
\begin{enumerate}
 \renewcommand{\theenumi}{\alph{enumi}}
\item The Axiom of countable Choice as in $\QFAC^{0,1}$.\label{CC}
\item K\"onig's infinity lemma as in $\KINL$.\label{KINL}
\item The Axiom of finite Choice as in Principle \ref{gohkuku}.\label{gohi}
%\item The previous item restricted to singleton sets. 
\end{enumerate}
\end{thm}
\begin{proof}
Assume $\KINL$ and let $\varphi$ be quantifier-free and such that $(\forall n\in \N)(\exists f\in \N^{\N})\varphi(n, f)$.  
Define $X$ as the set of $f\in \N^{\N}$ with $\varphi(f(0), f(1)*f(2)*\dots)$ and put $f=_{X}g$ if $f(0)=g(0)$.  
The set $E_{n}=\{ f\in X:  f(0)=n \}$ is non-empty and finite as in Definition \ref{trukkefoor}.
%Define $X$ as the (trivially coded) set of $w^{1^{*}}$ such that $(\forall i<|w|)\varphi(i, w(i))$ and put $w=_{X} v$ if $|w|=|v|$.
%The set $E_{n}=\{w\in X: |w|=n \}$ is non-empty by $\IND_{2}$ and finite as in Definition \ref{trukkefoor}.
%Define the binary relation $R$ by $wRv$ if $|v|=|w|+1$.  By $\Korg_{3}$, there is a sequence $(w_{n})_{n\in \N}$ with $w_{n}Rw_{n+1}$.
Define the binary relation $R$ by $fRg$ if $f(0)+1=g(0)$.  By $\KINL$, there is a sequence $(f_{n})_{n\in \N}$ with $f_{n}Rf_{n+1}$.
By definition, this sequence yields the required choice function for $\QFAC^{0,1}$.   This also yields that \eqref{gohi} $\di$  \eqref{CC}, where the reversal is trivial.   

\smallskip

To prove $\KINL$, we shall show the following (using induction) for all $n\in \N$:
\be\label{truf}
(\exists w^{1^{*}})\big[|w|=n \wedge (\forall i<|w|)(w(i)\in E_{i})\wedge (\forall j<|w|-1)(w(j)Rw(j+1))   \big].
\ee
Then apply $\QFAC^{0,1}$ to obtain $(w_{n})_{n\in \N}$.  With the latter sequence, one readily\footnote{Let $\sigma \in T$ hold in case $\sigma$ is an initial segment of $w_{n}$ for some $n\in \N$.} builds a finitely branching tree (in the sense of second-order RM).  Now use K\"onig's tree lemma (provided by $\ACA_{0}$; see \cite{simpson2}*{III.7.2}) to obtain the path required by $\KINL$.  To prove \eqref{truf} for all $n\in \N$, consider the formula $\varphi(n)$ defined as
\[
(\forall x\in E_{n})(\exists w^{1^{*}})\left[\begin{array}{c} |w|=n+1 \wedge (\forall i<|w|)(w(i)\in E_{i})\wedge x=w(n)\\
\wedge  (\forall j<|w|-1)(w(j)Rw(j+1)  )\end{array}\right].
\]
Now use $\IND_{3}$ to establish $(\forall n\in \N)\varphi(n)$, which implies \eqref{truf} for all $n\in \N$.
\end{proof}
The previous theorem is significant for hyperarithmetical analysis (see e.g.\ \cites{skore2, skore3}) as follows:  the finite choice principle $\FSAC$ from \cite{gohzeg} is essentially the second-order version of Principle~\ref{gohkuku}.  
However, $\FSAC$ is strictly weaker than $\SAC$, while Principle \ref{gohkuku} is equivalent to $\QFAC^{0,1}$ and $\ACAo+\QFAC^{0,1}$ is conservative over $\SAC$ (\cite{hunterphd}*{\S2}).  
Thus, the logical strength of $\FSAC$ depends on its second-order formulation: Definition \ref{trukkefoor} yields a more general principle.  
There should be more examples of this phenomenon.    
%as it suggest that, when formulated for $(X, =_{X})$, many theorems of hyperarithmetical analysis
%yield equivalences involving $\QFAC^{0,1}$.  Since $\ACAo+\QFAC^{0,1}$ is still a conservative extension of $\SAC$, the former can
%be said to be in the range of hyperarithmetical analysis.  Nonetheless, any THA between $\SAC$ and $\WSAC$ seems to yield an equivalence involving $\QFAC^{0,1}$
%when formulated for $(X, =_{X})$. 

\smallskip

%Finally, one readily obtains an equivalence as follows.  
%\begin{rem}[$\ACAo$]\rm
%Show that $\Korg_{3}\asa \QFAC^{0,1}$.  Use the proof of Theorem \ref{wood2} if necessary.  
%\end{rem}
In conclusion, finiteness as in Definition \ref{trukkefoor} yields equivalences between K\"onig's infinity lemma and $\QFAC^{0,1}$, and the same presumably for other statements of hyperarithmetical analysis.  
%Other such equivalences should be possible, e.g.\ involving Halin's infinite ray theorem from \cite{skore3}.   
% which is more general than other versions of 

%\subsubsection{Graphs}
%X

%Other graph theory? Like Hirst theorems or Shore's ray thm?

%\section{Conclusion}\label{after}
%The above results can be interpreted in various ways. One could write an iconoclastic pamphlet (and presumably post it to the FOM mailing list) entitled 
%\begin{center}
%\emph{J'Accuse\dots! Lettre au Pr\'esident de la R\'everse Math\'ematique}
%\end{center}
%pointing out the problem of coding in second-order arithmetic in wording bordering on the revolutionary.    
%
%\smallskip
%
%While such course of action may seem tempting and even expected of the author, our final judgement goes in the polar opposite direction.  
%Indeed, we believe that the above results suggest we should take Friedman's \emph{strict reverse mathematics} seriously.  

\end{document}